\documentclass[ECP,preprint]{ejpecp}

\SHORTTITLE{Finite-entropy criterion for the entropic CCLT}
\TITLE{A Finite-Entropy Criterion for the Entropic Conditional Central Limit Theorem
\support{Supported by the National Key R\&D Program of China, No.~2023YFA1009603.}}

\AUTHORS{%
  Tong Ye\footnote{Academy of Mathematics and Systems Science, Chinese Academy of Sciences,
Beijing, 100190, China.
    \EMAIL{yetong22@mails.ucas.ac.cn}}
  \and
 Liu-Quan Yao\footnote{Academy of Mathematics and Systems Science, Chinese Academy of Sciences,
Beijing, 100190, China. \EMAIL{yaoliuquan20@mails.ucas.ac.cn}}
\and  Shuai Yuan\footnote{Academy of Mathematics and Systems Science, Chinese Academy of Sciences,
Beijing, 100190, China. \EMAIL{yuanshuai2020@amss.ac.cn}}\and 
Guanghui Wang\footnote{School of Mathematics, Shandong University, Jinan, 250100, China. \EMAIL{ghwang@sdu.edu.cn}}
}


\AMSSUBJ{60F05}
\AMSSUBJSECONDARY{94A17}

\hypersetup{
  pdftitle={A Finite-Entropy Criterion for the Entropic Conditional Central Limit Theorem},
  pdfauthor={Tong Ye, Liu-Quan Yao, Shuai Yuan, Guanghui Wang}
}

\ABSTRACT{We prove a finite-entropy criterion for the entropic conditional central limit theorem. 
Let $(\xi_i,\eta_i)_{i\geq 1}$ be independent copies of a pair $(\xi,\eta)$, and set
$W_n=n^{-1/2}\sum_{i=1}^n \xi_i$ and $\boldsymbol{\eta}_n=(\eta_1,\ldots,\eta_n)$.
Under the assumptions that $\mathbb{E}\operatorname{Var}(\xi\mid\eta)<\infty$ and that 
the conditional law of $\xi$ given $\eta$ is absolutely continuous almost surely, we show that
$\mathbb{E}h(W_n\mid\boldsymbol{\eta}_n)$ converges to the Gaussian entropy 
$\frac12\log(2\pi e\sigma^2)$, where $\sigma^2=\mathbb{E}\operatorname{Var}(\xi\mid\eta)$, 
if and only if $\mathbb{E}h(W_{n_0}\mid\boldsymbol{\eta}_{n_0})>-\infty$ for some $n_0$.
The main technical ingredient is a continuity theorem for Fisher information under Gaussian smoothing, 
which allows us to replace the finite expected conditional Fisher-information assumption 
by a necessary and sufficient finite-entropy condition.}

\begin{document}

\section{Introduction}
The entropic central limit theorem, initiated by Linnik \cite{Linnik} 
and developed in its modern form by Barron \cite{barron1986entropy}, 
strengthens the classical central limit theorem by replacing weak convergence
 with convergence in relative entropy. 
For independent and identically distributed random variables $X_1,X_2,\ldots$ with
 mean zero and variance $\sigma^2$, 
Barron's theorem \cite{barron1986entropy} gives conditions under which the normalized sums $W_n=n^{-1/2}\sum_{i=1}^n X_i$
satisfy
\begin{equation}
    D(W_n\Vert Z_{\sigma^2})\to 0,
\end{equation}
where $Z_{\sigma^2}$ is Gaussian with variance $\sigma^2$. 
Equivalently, the differential entropies converge to the Gaussian entropy
\begin{equation}
	h(W_n)\to \frac12\log(2\pi e\sigma^2).
\end{equation}
Thus, together with Pinsker's inequality \cite{pinsker1964information}, this convergence implies the 
classical central limit theorem.

In a related direction, Johnson and Barron \cite{johnson2004fisher} proved 
a central limit theorem for Fisher information using projection methods.
Under a finite Poincaré-constant assumption, they also obtained convergence 
rates for Fisher information. Beyond the continuous setting, Gavalakis and 
Kontoyiannis \cite{gavalakis2024entropy} established an entropic CLT for 
discrete random variables via Bernoulli-part decomposition.

Conditional limit theorems arise naturally when the summands are observed together with a random 
environment or state variable. 
Necessary and sufficient conditions for conditional central 
limit theorems have been studied in several probabilistic settings, 
but the entropic conditional setting requires different information-theoretic tools. 
Given a pair $(\xi,\eta)$, let $(\xi_i,\eta_i)_{i\geq1}$ be independent copies and set
\begin{equation}
    W_n=\frac{\xi_1+\cdots+\xi_n}{\sqrt n},
    \quad
    \boldsymbol{\eta}_n=(\eta_1,\ldots,\eta_n).
\end{equation}
The entropic conditional central limit theorem concerns the asymptotic behavior of $h_n=\mathbb E h(W_n\mid\boldsymbol\eta_n)$.
Ma et al. \cite{ma2024entropic} recently proved such a theorem under a finite expected conditional Fisher-information condition.

The purpose of this note is to identify a necessary and sufficient finite-entropy condition for this convergence.
Assume that
\begin{equation}
	    \sigma^2=\mathbb E\operatorname{Var}(\xi\mid\eta)<\infty
\end{equation}
and that the conditional law of $\xi$ given $\eta$ is absolutely continuous almost surely. 
We prove that
\begin{equation}
    h_n\to \frac12\log(2\pi e\sigma^2)	
\end{equation}
if and only if
\begin{equation}
   h_{n_0}>-\infty	
\end{equation}
for some $n_0\geq1$. 
This gives a conditional analogue of Barron's finite-entropy criterion \cite{barron1986entropy} for the entropic central limit theorem.

The main new ingredient is a continuity result for Fisher information under Gaussian smoothing. 
More precisely, on a compact class of Gaussian-smoothed densities with uniformly integrable centered second moments, Fisher information is continuous with respect to the $L^1$ topology. 
This continuity theorem complements Bobkov's lower semicontinuity result \cite{Upper} 
by establishing continuity on a Gaussian-smoothed, tail-controlled compact class. 
It also allows us to pass from entropy convergence after Gaussian regularization 
to convergence of the conditional Fisher information. Combining this with de Bruijn's identity \cite{stam1959some} yields the stated entropic conditional central limit theorem.

The paper is organized as follows. 
Section \ref{sec:main} states the finite-entropy criterion and gives a simple example. 
Section \ref{sec:fisher} proves the Fisher-information continuity theorem and derives the conditional Fisher-information convergence proposition. 
Section \ref{sec:proof} proves the main theorem.

	\section{Main result }\label{sec:main}
  In this section, we state the main result of the paper, a finite-entropy criterion for the entropic conditional central limit theorem. 
We then discuss its relation to previous assumptions and give an illustrative example.

	\begin{theorem}[Finite-entropy criterion for the entropic CCLT]\label{th:main}
	Let $(\xi,\eta )$ be a pair of random variables. 
    Let $\{(\xi_i,\eta_i)\}_{i\geq 1}$ be independent copies of $(\xi,\eta)$.
	  Define
	  \begin{equation}
					W_n = \frac{\xi_1 + \xi_2 + \cdots + \xi_n}{\sqrt{n}}, 
			\quad \boldsymbol{\eta}_{n}= (\eta_1, \ldots, \eta_n).	
	  \end{equation}
			Assume that
\begin{equation}
			\sigma^{2} \triangleq\mathbb{E}\,\operatorname{Var}(\xi  \vert \eta) < \infty, 
		\quad (\xi \vert \eta) \ll \mathcal{L}, \ P_{\eta} \text{-a.s.},
\end{equation}
			where $\mathcal{L} $ denotes Lebesgue measure and the notation $\nu \ll \mathcal{L}$ indicates that $\nu$ is 
			absolutely continuous with respect to  $\mathcal{L}$. Define
	\begin{equation}
	h_{n}= \mathbb{E}\, h\!\left(W_{n} \vert \boldsymbol{\eta}_{n}\right).		
	\end{equation}		
			Then the convergence
\begin{equation}\label{2}
			\lim_{n \to \infty} h_n = \frac{1}{2} \log (2\pi e \sigma^2)
\end{equation}
holds if and only if there exists $n_0 \geq 1$ such that
\begin{equation}
	h_{n_{0}}= \mathbb{E}\, h\!\left(W_{n_0} \vert \boldsymbol{\eta}_{n_{0}}\right) > -\infty.
\end{equation}			

	\end{theorem}

\begin{remark}[Comparison with previous assumptions]\label{rem:comparison}
 Compared with the entropic conditional central limit theorem of Ma et al. \cite{ma2024entropic}, 
Theorem \ref{th:main} replaces the finite expected conditional Fisher-information assumption by the finite-entropy condition
\begin{equation}
     \mathbb E h(W_{n_0}\mid\boldsymbol{\eta}_{n_0})>-\infty
\end{equation}
for some $n_0\geq1$. 
It also assumes directly that $\mathbb E\operatorname{Var}(\xi\mid\eta)<\infty$, rather than the stronger condition $\operatorname{Var}(\xi)<\infty$. 
In particular, $\mathbb E h(\xi\mid\eta)>-\infty$ is a sufficient special case, corresponding to $n_0=1$.
\end{remark}

\begin{remark}[Conditional relative entropy convergence]\label{rem:relative-entropy}
The conclusion of Theorem \ref{th:main} also implies convergence in conditional relative entropy. 
Indeed,
\begin{align}
\mathbb{E}D(W_n\mid \boldsymbol{\eta}_n)
&=
\frac12\mathbb{E}\log\left(2\pi e\,\operatorname{Var}(W_n\mid \boldsymbol{\eta}_n)\right)
-
\mathbb{E}h(W_n\mid \boldsymbol{\eta}_n) \\
&\leq
\frac12\log\left(2\pi e\,\mathbb{E}\operatorname{Var}(W_n\mid \boldsymbol{\eta}_n)\right)
-
\mathbb{E}h(W_n\mid \boldsymbol{\eta}_n) \\
&=
\frac12\log(2\pi e\sigma^2)
-
\mathbb{E}h(W_n\mid \boldsymbol{\eta}_n).
\end{align}
Hence Theorem \ref{th:main} yields
\begin{equation}
    \lim_{n\to\infty}\mathbb{E}D(W_n\mid \boldsymbol{\eta}_n)=0.	
\end{equation}
When $\xi$ is independent of $\eta$, this recovers Barron's entropic central limit theorem \cite{barron1986entropy}. 
By Pinsker's inequality \cite{pinsker1964information}, this convergence is stronger than the corresponding conditional weak convergence.
		
	\end{remark}

\begin{example}[State-dependent non-Gaussian noise]\label{ex:state-dependent-noise}
Let $\eta_i\sim\mathrm{Bernoulli}(1/2)$, $i\in\mathbb N^+$, and suppose that
\[
\xi_i\mid(\eta_i=0)\sim\mathrm{Unif}(-1,1),
\qquad
\xi_i\mid(\eta_i=1)\sim\mathrm{Unif}(-2,2).
\]
Set $W_n=n^{-1/2}\sum_{i=1}^n\xi_i$ and $\boldsymbol\eta_n=(\eta_1,\ldots,\eta_n)$.
Then the conditional densities are
\begin{equation}
	f_{\xi \vert \eta}(x\vert 0)=\frac{1}{2} \boldsymbol{1}_{(-1,1)}(x), \quad
	f_{\xi \vert \eta}(x\vert 1)=\frac{1}{4} \boldsymbol{1}_{(-2,2)}(x).
\end{equation}
Moreover, $\operatorname{Var}(\xi\mid\eta=0)=1/3$ and $\operatorname{Var}(\xi\mid\eta=1)=4/3$, so $\sigma^2=\mathbb E\operatorname{Var}(\xi\mid\eta)=5/6$.
The finite-entropy condition holds since $h(\xi\mid\eta=0)=\log2$ and $h(\xi\mid\eta=1)=\log4$, hence $h_1=\mathbb Eh(\xi\mid\eta)=\frac32\log2>-\infty$.
Therefore, Theorem \ref{th:main} gives
\[
\lim_{n\to\infty}\mathbb Eh(W_n\mid\boldsymbol\eta_n)
=
\frac12\log\left(2\pi e\cdot\frac56\right)
=
\frac12\log\left(\frac{5\pi e}{3}\right).
\]
This example may be interpreted as a state-dependent noise model, where the channel 
state is $\eta$ and the distribution of the noise $\xi$ depends on the state. 
Although the conditional noise distribution varies with the state and is non-Gaussian, 
the average conditional entropy of the normalized sum converges to the entropy of a Gaussian 
random variable with variance $\sigma^2=\mathbb{E}\operatorname{Var}(\xi\mid\eta)$.
\end{example}

 \section{Fisher-information continuity and convergence}\label{sec:fisher}

In this section, we prove the main technical ingredients used in the proof of Theorem
 \ref{th:main}. The first result is a continuity theorem for Fisher information under Gaussian smoothing. 
The second result applies this continuity theorem to obtain convergence of the conditional Fisher information.

\subsection{Continuity of Fisher information}
For a density $f$, denote
\begin{equation}
	L^1(f\,dx)\triangleq L^1(\mathbb{R},f(x)\,dx),
    \qquad
    L^1(dx)\triangleq L^1(\mathbb{R},dx).
\end{equation}
For a random variable $X$, define
\begin{equation}
    L_X(R)
    \triangleq
    \mathbb{E}\left[
    (X-\mathbb{E}X)^2
    \mathbf{1}_{\{|X-\mathbb{E}X|\geq R\}}
    \right].
\end{equation}
For $a>0$ and a function $l$ satisfying $l(R)\to 0$ as $R\to\infty$, define
\begin{equation}
    \mathcal{I}_{a,l}
    \triangleq
    \left\{
    f_{X_0+G_a}:
    L_{X_0+G_a}(R)\leq l(R),\ \forall R>0,\ 
    \mathbb{E}X_0=0
    \right\},
\end{equation}
where $G_a\sim\mathcal N(0,a)$, $X_0$ is independent of $G_a$, and 
$f_{X_0+G_a}$ denotes the density of $X_0+G_a$. 
Let
\begin{equation}
   \overline{\mathcal I_{a,l}}
    \triangleq
    \operatorname{cl}_{L^1_{1+x^2}}(\mathcal I_{a,l}).
\end{equation}
The next theorem is the main technical ingredient of this section.
It states that the Gaussian-smoothed, tail-controlled class $ \overline{\mathcal I_{a,l}}$ is compact in $L^1$ 
and that Fisher information is 
 continuous on this class with respect to $L^1$ convergence.
\begin{theorem}[Fisher-information continuity]\label{thm:fisher-continuity}
The set $\overline{\mathcal I_{a,l}}$ is compact in $L^1_{1+x^2}$, and hence also compact in $L^1$.
Moreover, the Fisher-information functional
\begin{equation}
    p\mapsto J(p)
\end{equation}
is continuous on $\overline{\mathcal I_{a,l}}$ with respect to the $L^1$ topology.
Equivalently, if $p_n,p\in \overline{\mathcal I_{a,l}}$ and
\begin{equation}
    \Vert p_n-p\Vert_{L^1}\rightarrow 0, \, n \rightarrow \infty. 
\end{equation}
then
\begin{equation}
    J(p_n)\to J(p).
\end{equation}
\end{theorem}

The proof of Theorem \ref{thm:fisher-continuity} is given in Appendix \ref{app:fisher-continuity}.
The main point is that Gaussian smoothing provides uniform regularity, while the tail condition encoded by $l$ gives compactness in the weighted space $L^1_{1+x^2}$. 
These two ingredients allow the Fisher information functional to be continuous 
on $\overline{\mathcal I_{a,l}}$, although Fisher information is not continuous under general $L^1$ convergence. 

\subsection{Conditional Fisher-information convergence}

We now use Theorem \ref{thm:fisher-continuity} to prove the conditional 
Fisher-information convergence needed for the proof of Theorem \ref{th:main}. 
For random variables $X$ and $Y$, define
\[
    L_{X\mid Y}(R)
    \triangleq
    \mathbb E\left[
    (X-\mathbb E[X\mid Y])^2
    \mathbf 1_{\{|X-\mathbb E[X\mid Y]|\geq R\}}
    \mid Y
    \right].
\]

\begin{proposition}[Conditional Fisher-information convergence]
\label{prop:conditional-fisher}\hfill\break
Let $\{(X_n,Y_n)\}_{n\geq 1}$ be a sequence of pairs of random variables such that, for each $n$, 
the conditional distribution of $X_n$ given $Y_n$ is absolutely continuous with respect to Lebesgue measure almost surely. 
Suppose that
\begin{equation}
	    X_n=X_{0,n}+G_{a,n},
\end{equation}
where $G_{a,n}\sim\mathcal N(0,a)$ for some $a>0$, and $G_{a,n}$ is independent of $(X_{0,n},Y_n)$.

Assume that there exists a nonnegative function $l:[0,\infty)\to[0,\infty)$ such that $l(R)\to0$ as $R\to\infty$ and
\begin{equation}
    \mathbb{E}L_{X_n\mid Y_n}(R)\leq l(R),
    \qquad \forall n\in\mathbb N^+,\ R\geq 0.	
\end{equation}
Let $(X_n',Y_n')$ be an independent copy of $(X_n,Y_n)$. 
Suppose further that the limits
\begin{equation}\label{3.3}
    \lim_{n\to\infty}\mathbb{E}h(X_n\mid Y_n)
    \quad\text{and}\quad
    \lim_{n\to\infty}\mathbb{E}\operatorname{Var}(X_n\mid Y_n)
\end{equation}
exist, and that, for some $\lambda\in(0,1)$,
\begin{equation}
    \left|
    \mathbb{E}h(X_{n+1}\mid Y_{n+1})
    -
    \mathbb{E}h(X_n\mid Y_n)
    \right|
    =
    \mathbb{E}\left[
    h(X_n*_{\lambda}X_n'\mid Y_n,Y_n')
    -
    h(X_n\mid Y_n)
    \right],
\end{equation}
where
\begin{equation}
    X_n*_{\lambda}X_n'
    \triangleq
    \lambda X_n+\sqrt{1-\lambda^2}\,X_n' .
\end{equation}
Then
\begin{equation}
    \lim_{n\to\infty}\mathbb{E}J(X_n\mid Y_n)
    =
    \frac{1}{\sigma^2},
\end{equation}
where
\begin{equation}
    \sigma^2
    =
    \lim_{n\to\infty}
    \mathbb{E}\operatorname{Var}(X_n\mid Y_n).
\end{equation}
\end{proposition}

\begin{proof}
  Let $J_{n}=\mathbb{E}J(X_{n}\vert Y_{n}), V_{n}=\mathbb{E}\operatorname{Var}(X_{n}\vert Y_{n})$.  On the one hand,
  By the Cram\'{e}r--Rao inequality, for almost every $y_n$,
  \begin{equation}
        J(X_n\mid Y_n=y_n)\operatorname{Var}(X_n\mid Y_n=y_n)\geq 1.
  \end{equation}
Together with the Cauchy--Schwarz inequality, this gives
\begin{equation}
        V_nJ_n
    \geq
    \left(
    \mathbb{E}\sqrt{
    J(X_n\mid Y_n)\operatorname{Var}(X_n\mid Y_n)
    }
    \right)^2
    \geq 1, \, \forall n\in\mathbb{N}^{+}.
\end{equation}
    This implies $\liminf_{n\rightarrow \infty}J_{n} \geq \frac{1}{\sigma^{2}}$. On the other hand, 
	for any $\alpha>0$, it follows from Proposition 5.16 of \cite{ma2024entropic} that we can 
	find a function l decreasing to 0 such that the set 
	\begin{equation}
		U_{n} =\lbrace y_{n}:L_{X_{n}\vert y_{n}}  \leq l \rbrace,
	\end{equation}
	satisfies $\mathbb P(U_n)\geq 1-\alpha$. By Theorem \ref{thm:fisher-continuity}, the Fisher information is 
	continuous on $\overline{\mathcal I_{a,l}}$ with respect to the $L^{1}$ norm. Therefore,  using Pinsker's
	 inequality \cite{pinsker1964information}, 
     we obtain that, for any $\epsilon>0$, there exists $\delta(\epsilon,a,l)>0$ such that for any
	 $p,q \in \mathcal{I}_{a,l} $, 
	 \begin{equation}\label{3.12}
		KL(p\Vert q) \leq \delta \Longrightarrow \vert J(p)-J(q)\vert \leq \epsilon. 
	 \end{equation}
    Define 
	\begin{align}
		R_{n} &= \lbrace y_{n}:D(X_{n}\vert y_{n}) \leq \delta  \rbrace, \\
		W_{n} &= \lbrace y_{n}:\vert \operatorname{Var}(X_{n}\vert y_{n}) -\sigma^{2} \vert \leq \epsilon \rbrace.
	\end{align}
  Let $T_{n}=U_{n} \bigcap R_{n}\bigcap W_{n}$. By \eqref{3.12} and
	\begin{equation}
		D(X_{n}\vert y_{n})= \frac{1}{2}\ln 2\pi e \operatorname{Var}(X_{n}\vert y_{n}) -h(X_{n}\vert y_{n}),
	\end{equation}
     we obtain that 
	\begin{equation}
		\vert J(X_{n}\vert y_{n}) -\frac{1}{\operatorname{Var}(X_{n}\vert y_{n})} \vert  \leq \epsilon, \forall y_{n}\in T_{n}.
	\end{equation}
    Therefore, 
	 \begin{align}
		J_{n} &= \mathbb{E} \lbrack J(X_{n}\vert Y_{n}) \mathbf{1}_{T_{n}} \rbrack + 
		\mathbb{E} \lbrack J(X_{n}\vert Y_{n}) \mathbf{1}_{T_{n}^{C}} \rbrack \\
		&\leq  
		\mathbb{E}\left[
\left(
\epsilon+
\frac{1}{\operatorname{Var}(X_n\mid Y_n)}
\right)
\mathbf{1}_{T_n}
\right]+ \frac{1}{a}\mathbb P(T_{n}^{C}),
	 \end{align} 
	where the last inequality holds follows from
	\begin{equation}
		J(X_{n}\vert Y_{n}) = J(X_{0,n}+G_{a,n}\vert Y_{n}) \leq J(G_{a,n})=\frac{1}{a}.
	\end{equation}  
Note that for any $y_{n}\in T_{n}$, we have 
\begin{equation}
	\sigma^{2} -\epsilon \leq \operatorname{Var}(X_{n}\vert y_{n}) \leq \sigma^{2}+\epsilon.
\end{equation}
This implies 
\begin{align}
	J_{n} &\leq \epsilon+ \frac{1}{\sigma^{2}-\epsilon} +\frac{\mathbb P(U_{n}^{C})+\mathbb P(R_{n}^{C})
    +\mathbb P(W_{n}^{C})}{a}\\
  &\leq  \epsilon+ \frac{1}{\sigma^{2}-\epsilon} +\frac{\alpha+\mathbb P(R_{n}^{C})+ \mathbb P(W_{n}^{C})}{a}.
\end{align}
 From Lemma 5.1 of \cite{ma2024entropic} and equation (\ref{3.3}), we obtain
\begin{align}
	\lim_{n\rightarrow \infty}\mathbb P(R_{n}^{c})=&0\\
	\lim_{n\rightarrow \infty}\mathbb P(W_{n}^{c})=&0.
\end{align}
Therefore, 
\begin{equation}
	\limsup_{n\rightarrow \infty } J_{n} \leq \epsilon+\frac{1}{\sigma^{2}-\epsilon}+\frac{\alpha}{a}.
\end{equation}
Since $\alpha$ and $\epsilon$ are arbitrary, we have 
\begin{equation}    
	\lim_{n\rightarrow\infty}J_{n}= \frac{1}{\sigma^{2}}.
\end{equation}
\end{proof}

\section{Proof of Theorem \ref{th:main}} \label{sec:proof}
Now we can prove Theorem \ref{th:main}. The necessity is proved by contradiction. 
To prove sufficiency, we first verify the assumptions of Proposition \ref{prop:conditional-fisher}. Then we apply its result together with 
de Bruijn's identity \cite{stam1959some} to
 obtain the theorem's conclusion for a subsequence of $\{h_n\}$. Finally we prove the convergence of $\{h_n\}$  via subadditivity.

\begin{proof}
  We first prove necessity by contradiction. Assume that $h_n=-\infty$ for all $n\in\mathbb N^+$.
Then
\begin{equation}\label{4.2}
	\lim_{n\rightarrow \infty}h_{n}=-\infty.
\end{equation}
However, it follows from the assumptions of Theorem \ref{th:main} that since 
\begin{equation}
  (\xi\vert \eta) \ll \mathcal{L}, \ P_{\eta} \text{-a.s.}.
\end{equation}
we have 
\begin{equation}
	\sigma^{2} =\mathbb{E}\,\mathrm{Var}(\xi \vert \eta) >0.
\end{equation}
Therefore, from the condition of Theorem \ref{th:main}, 
\begin{equation}
	\lim_{n\rightarrow \infty } h_{n} =\frac{1}{2} \log (2\pi e \sigma^{2}) >- \infty.
\end{equation}
This yields a contradiction with equation (\ref{4.2}). 

Next, we prove the sufficiency part.  First, by  de Bruijn's identity \cite{stam1959some}, we have 
\begin{equation}
	h(X+\sqrt{s}G)-h(X)=\int_{0}^{s}\frac{J(X+\sqrt{t}G)}{2}dt,
\end{equation}
where $G\sim\mathcal N(0,1)$ is independent of $X$.   Therefore, 
\begin{equation}\label{4.7}
	 \mathbb{E}\, h\!\left(W_{n_{0}} \vert \boldsymbol{\eta}_{n_{0}}\right) 
	 =\mathbb{E}\, h\!\left(W_{n_{0}} +\sqrt{s}G\vert \boldsymbol{\eta}_{n_{0}}\right)
	 -\mathbb{E}\int_{0}^{s}
	 \frac{J(W_{n_{0}}+\sqrt{t}G\vert \boldsymbol{\eta}_{n_{0}} ) }{2}dt.
\end{equation}
Since
\begin{equation}
	h(X) \leq \frac{1}{2} \log (2\pi e \operatorname{Var}(X)),
\end{equation}
we have 
\begin{align}
	\mathbb{E}\, h\!\left(W_{n_{0}} +\sqrt{s}G\vert \boldsymbol{\eta}_{n_{0}}\right) 
	&\leq \mathbb{E}\frac{1}{2}\log (2\pi e \operatorname{Var}(W_{n_{0}} +\sqrt{s}G\vert \boldsymbol{\eta}_{n_{0}}))\\
	&\leq \frac{1}{2}\log (2\pi e \mathbb{E}\operatorname{Var}(W_{n_{0}} +\sqrt{s}G\vert \boldsymbol{\eta}_{n_{0}}))\label{4.10}\\
	&\leq \frac{1}{2}\log (2\pi e (\sigma^{2}+s))< \infty\label{4.11}.
\end{align}
where the inequality in (\ref{4.10})  follows from Jensen's inequality, and the inequality in (\ref{4.11}) follows from 
the independence of G and $(W_{n_{0}},\boldsymbol{\eta}_{n_{0}})$.
Since
\begin{equation}
	\mathbb{E}\, h\!\left(W_{n_{0}} \vert \boldsymbol{\eta}_{n_{0}}\right)  > -\infty,
\end{equation}
we have 
\begin{equation}
	\mathbb{E}\int_{0}^{s}
	 \frac{J(W_{n_{0}}+\sqrt{t}G\vert \boldsymbol{\eta}_{n_{0}} ) }{2}dt <\infty.
\end{equation}
Denote 
\begin{equation}
	X_{m}= W_{2^{m}n_{0}}+\sqrt{t}G, \quad Y_{m}= \boldsymbol{\eta}_{2^{m}n_{0}}.
\end{equation}
In order to apply Proposition \ref{prop:conditional-fisher}, we need to verify that its assumptions hold.
Let  
\begin{equation}
	\bar W_n
=
\frac{1}{\sqrt n}
\sum_{k=1}^n
\left(\tilde \xi_k-\mathbb E[\tilde \xi_k\mid \eta_k]\right),
\end{equation} 
where $\tilde{\xi}_{k} = \xi_{k}+\sqrt{t}G_{k},$ and $G_{k} \stackrel{i.i.d.}{\sim} \mathcal{N}(0,1) $ is 
independent of $\{\xi_{i},\eta_{i}\}_{i \in \mathbb{N}^{+}}$. 
By the argument in Appendix G of \cite{ma2024entropic},  we obtain $\{ \bar{W_{n}^{2}} \}$ is uniformly integrable. Let
\begin{equation}
	l(R) = \sup_{n}\mathbb{E} (\bar{W_{n}^{2}} 1_{\vert  \bar{W_{n}}\vert \geq R}),
\end{equation}
So $l(R)$ is a decreasing function and $l(R) \rightarrow 0, $ as $R \rightarrow \infty $. Therefore
\begin{equation}
\mathbb{E}L_{X_m\mid Y_m}(R)\leq l(R),
\quad \forall m\in\mathbb N^+,\ R\geq0.
\end{equation}
In addition,  
\begin{equation}
	\mathbb{E}\operatorname{Var}(X_{m}\vert Y_{m})=	\mathbb{E} \operatorname{Var}(W_{2^{m}n_{0}}+\sqrt{t}G \vert \boldsymbol{\eta}_{2^{m}n_{0}})=\sigma^{2}+t,
\end{equation}
then by Proposition \ref{prop:conditional-fisher} 
\begin{equation} \label{4.19}
	\lim_{m\rightarrow \infty}\mathbb{E}J(X_{m}\vert Y_{m}) =	\lim_{m\rightarrow \infty}\mathbb{E} J(W_{2^{m}n_{0}}
	+\sqrt{t}G\vert \boldsymbol{\eta}_{2^{m}n_{0}}) =\frac{1}{\sigma^{2}+t}.
\end{equation}
From \cite{ma2024entropic}, we obtain
\begin{equation}\label{4.20}
	\lim_{m\rightarrow \infty} \mathbb{E}\, h\!\left(W_{2^{m}n_{0}} +\sqrt{t}G\vert 
	\boldsymbol{\eta}_{2^{m}n_{0}}\right)  = \frac{1}{2}\log (2\pi e (\sigma^{2}+t)).
\end{equation}
And it follows from Fubini's Theorem  that %
\begin{equation}\label{4.21}
	\mathbb{E}\int_{0}^{s}
	 \frac{J(W_{n_{0}}+\sqrt{t}G\vert \boldsymbol{\eta}_{n_{0}} ) }{2}dt = \int_{0}^{s} \mathbb{E}
	 \frac{J(W_{n_{0}}+\sqrt{t}G\vert \boldsymbol{\eta}_{n_{0}} ) }{2}dt <\infty.
\end{equation}
 Combining  \eqref{4.19} and  \eqref{4.20}, it suffices to verify the  conditions of the %
 dominated convergence theorem. Since
 \begin{equation}
	J(\frac{X+Y}{\sqrt{2}}) \leq \frac{J(X)+J(Y)}{2},
 \end{equation} 
then 
\begin{equation}
	J(W_{2^{m}n_{0}}+\sqrt{t}G\vert \boldsymbol{\eta}_{2^{m}n_{0}}) \leq
	\frac{\sum_{i=1}^{2^{m}}J(W_{in_{0}+1}^{(i+1)n_{0}}+\sqrt{t}G \vert\boldsymbol{\eta}_{in_{0}+1}^{(i+1)n_{0}})}{2^{m}}.   
\end{equation}
Taking expectations on both sides, we obtain 
\begin{align}
	\mathbb{E}J(W_{2^{m}n_{0}}+\sqrt{t}G\vert \boldsymbol{\eta}_{2^{m}n_{0}})&\leq 
	\frac{\sum_{i=1}^{2^{m}}\mathbb{E}J(W_{in_{0}+1}^{(i+1)n_{0}}+\sqrt{t}G \vert \boldsymbol{\eta}_{in_{0}+1}^{(i+1)n_{0}})}{2^{m}} \\
 &\leq \mathbb{E}J(W_{n_{0}}+\sqrt{t}G\vert \boldsymbol{\eta}_{n_{0}}) ,
\end{align}
where the second inequality follows from identical distribution.
Since  \eqref{4.20} holds,  by using  the dominated convergence theorem, we obtain
\begin{equation}
	\lim_{m\rightarrow \infty}\mathbb{E}\int_{0}^{s}
	 \frac{J(W_{2^{m}n_{0}}+\sqrt{t}G\vert \boldsymbol{\eta}_{2^{m}n_{0}} ) }{2}dt=\int_{0}^{s}\frac{1}{2(\sigma^{2}+t)}dt =\frac{1}{2}\log(\frac{\sigma^{2}+s}{\sigma^{2}})
\end{equation}
Then combining equations (\ref{4.7}) and (\ref{4.19}), we get 
\begin{equation}
	\lim_{m\rightarrow \infty}\mathbb{E}h(W_{2^{m}n_{0}}\vert  \boldsymbol{\eta}_{2^{m}n_{0}})= \frac{1}{2}\log (2\pi e \sigma^{2}). 
  \end{equation}
  Finally we only need to show that $h_{n}$ converges. By the entropy jump inequality \cite{barron1986entropy}, if $X$ and $Y$ are independent, then
\begin{equation}
    h\left(\lambda X+\sqrt{1-\lambda^2}\,Y\right)
    \geq
    \lambda^2 h(X)+(1-\lambda^2)h(Y),
    \quad \lambda\in[0,1].
\end{equation}
For $m,n\geq1$, define
\begin{equation}
	    W_{n+1}^{m+n}
    \triangleq
    \frac{\xi_{n+1}+\cdots+\xi_{m+n}}{\sqrt m},
    \quad
    \boldsymbol{\eta}_{n+1}^{m+n}
    \triangleq
    (\eta_{n+1},\ldots,\eta_{m+n}).
\end{equation}
Applying the entropy jump inequality conditionally on $\boldsymbol{\eta}_{m+n}$ gives
\begin{align}
    h(W_{m+n}\mid \boldsymbol{\eta}_{m+n})
    &=
    h\left(
    \frac{\sqrt n}{\sqrt{m+n}}W_n+
    \frac{\sqrt m}{\sqrt{m+n}}W_{n+1}^{m+n}
    \,\middle|\,\boldsymbol{\eta}_{m+n}
    \right) \\
    &\geq
    \frac{n}{m+n}h(W_n\mid \boldsymbol{\eta}_{m+n})
    +
    \frac{m}{m+n}h(W_{n+1}^{m+n}\mid \boldsymbol{\eta}_{m+n}) \\
    &=
    \frac{n}{m+n}h(W_n\mid \boldsymbol{\eta}_n)
    +
    \frac{m}{m+n}h(W_{n+1}^{m+n}\mid \boldsymbol{\eta}_{n+1}^{m+n}).
\end{align}
Taking expectations on both sides, we obtain
\begin{equation}
    h_{m+n}
    \geq
    \frac{n}{m+n}h_n+\frac{m}{m+n}h_m,
\end{equation}
or equivalently,
\begin{equation}
    (m+n)h_{m+n}\geq n h_n+m h_m.
\end{equation}
Note that if $X$ and $Y$ are independent, then
\begin{equation}
    h(X+Y)\ge h(X).
\end{equation}
For $n>n_0$, we have
\begin{equation}
    W_n
    =
    \sqrt{\frac{n_0}{n}}\,W_{n_0}
    +
    \sqrt{\frac{n-n_0}{n}}\,W_{n_0+1}^{n}.
\end{equation}
Conditionally on $\boldsymbol{\eta}_n$, the two terms on the
right-hand side are independent. Hence,
\begin{equation}
    h(W_n\mid\boldsymbol{\eta}_n)
    \ge
    h(W_{n_0}\mid\boldsymbol{\eta}_{n_0})
    +
    \frac12\log\frac{n_0}{n}.
\end{equation}
Since
\begin{equation}
    h_{n_0}
    =
    \mathbb E\,
    h\!\left(
        W_{n_0}\mid\boldsymbol{\eta}_{n_0}
    \right)
    >-\infty,
\end{equation}
taking expectations yields
\begin{equation}
    h_n
    =
    \mathbb E\,
    h(W_n\mid\boldsymbol{\eta}_n)
    \ge
    h_{n_0}
    +
    \frac12\log\frac{n_0}{n}
    >-\infty.
\end{equation}
Moreover, 
\begin{equation}
    h_n
    \le
    \frac12\log(2\pi e\sigma^2)
    <\infty.
\end{equation}
Thus $h_n$ is finite for every $n>n_0$.

Since the sequence $a_n\triangleq n h_n$ is superadditive, Fekete's
lemma yields
\begin{equation}
    \lim_{n\to\infty}\frac{a_n}{n}
    =
    \sup_{n\ge1}\frac{a_n}{n}.
\end{equation}
Equivalently,
\begin{equation}
    \lim_{n\to\infty}h_n
    =
    \sup_{n\ge1}h_n.
\end{equation}
Therefore,
\begin{equation}
    \lim_{n\to\infty} h_n
    =
    \frac12\log(2\pi e\sigma^2).
\end{equation}
\end{proof}
 
\appendix

\section{Proof of Theorem \ref{thm:fisher-continuity}}
\label{app:fisher-continuity}

\begin{proof}[Proof of Theorem \ref{thm:fisher-continuity}]
  Without loss of generality, we assume that $a\in (0,1)$. For any $f_{X} \in \mathcal{I}_{a,l}$, let
   \begin{align}
 X= X_{0}+G_{a} \stackrel{d}{=}X_{0} +\hat{G}_{\frac{a}{2}} + G_{\frac{a}{2}}  &\stackrel{d}{=} X_{0} +\hat{G}_{\frac{a}{2}}
 +\sqrt{\frac{a}{2}}G \\	
 & = e^{-t}Z +\sqrt{1-e^{-2t}} G ,
   \end{align}
  where $t =\frac{-\ln(1-\frac{a}{2})}{2}$, $Z=e^{t}(X_{0}+\hat{G}_{\frac{a}{2}})$ and $\hat{G}_{\frac{a}{2}}$ 
  and $G_{\frac{a}{2}}$ are i.i.d. $\sim \mathcal{N}(0,\frac{a}{2})$. Set
  \begin{equation}
	p^{*}_{t} Z = e^{-t}Z +\sqrt{1-e^{-2t}} G,
  \end{equation}
  where $\{ p_{t}^{*}: t\geq 0\}$ denotes the adjoint operator of the Ornstein-Uhlenbeck semigroup with %
  respect to the Lebesgue measure.  Since 
  \begin{equation}
	\Gamma \triangleq \lbrace \vert X_{0} \vert \geq 2R \rbrace \bigcap \lbrace \vert G_{a} \vert  \leq R \rbrace 
 \subset \lbrace  \vert X\vert  \geq R\rbrace,
\end{equation}
we have
\begin{align}
	l(R)&\geq L_{X}(R) \geq \mathbb{E} \lbrack X^{2} \mathbf{1}_{\Gamma}\rbrack \geq  \mathbb{E}\lbrack 
	(X^{2}_{0}+2X_{0}G_{a}) \mathbf{1}_{\Gamma}\rbrack\\
	&\geq  L_{X_{0}}(2R)P(\vert G_{a}\vert  \leq R)+2\mathbb{E}\lbrack 
	X^{2}_{0} 1_{\vert X_{0}\geq 2R \vert }\rbrack\mathbb{E}\lbrack 
	G_{a} 1_{\vert G_{a}\leq R \vert }\rbrack    \\
	&\geq L_{X_{0}}(2R)(1-\frac{a^{2}}{R^{2}}),\label{A.7}
\end{align}
where the last inequality follows from Chebyshev's inequality and the symmetry of the Gaussian density. Therefore, we obtain the upper bound for $L_{X_{0}}(R)$. Note that 
 $Z=e^t(X_0+\widehat G_{a/2})$, where $\widehat G_{a/2}\sim\mathcal N(0,a/2)$ is independent of $X_0$. 
For $R>0$, set $r=Re^{-t}/2$. Since
\[
    \{|Z|\geq R\}
    \subset
    \{|X_0|\geq r\}\cup\{|\widehat G_{a/2}|\geq r\},
\]
we obtain
\begin{align}
    L_Z(R)
    &=
    \mathbb{E}\left[Z^2\mathbf{1}_{\{|Z|\geq R\}}\right] \notag\\
    &\leq
    2e^{2t}
    \mathbb{E}\left[
    (X_0^2+\widehat G_{a/2}^{\,2})
    \left(
    \mathbf{1}_{\{|X_0|\geq r\}}
    +
    \mathbf{1}_{\{|\widehat G_{a/2}|\geq r\}}
    \right)
    \right] \notag\\
    &\leq
    2e^{2t}L_{X_{0}}(r)
    +
    2e^{2t}L_{X_{0}}(0)\mathbb{P}(|\widehat G_{a/2}|\geq r)
    +
    ae^{2t}\mathbb{P}(|X_0|\geq r)
    +
    2e^{2t}L_{\widehat G_{a/2}}(r) \notag\\
	& \leq  e^{2t}(a+2)L_{X_{0}}(r)+2e^{2t}L_{X_{0}}(0)L_{\hat{G}_{\frac{a}{2}}}(r).
\end{align}
 And for $R \in \lbrack 0,2e^{t} \rbrack$, we have
 \begin{equation}
	L_{Z}(R) \leq L_{Z}(0) \leq 2 e^{2t} L_{X_{0}}(0)+2 e^{2t}
	L_{\hat{G}_{\frac{a}{2}}}(0)= 2e^{2t}(L_{X_{0}}(0)+\frac{a}{2}).
 \end{equation}
Denote
\begin{equation}
    l_0'
    \triangleq
    \max\Bigg\{2e^{2t}\left(L_{X_0}(0)+\frac{a}{2}\right), e^{2t}(a+2)L_{X_0}(1)
    +2e^{2t}L_{X_0}(0)L_{\widehat G_{a/2}}(1)
    \Bigg\}.
\end{equation}
And define
\begin{equation}
    l'(R)=
    \begin{cases}
    l_0',
    & R\in [0,2e^t],\\[0.4em]
    e^{2t}(a+2)
    L_{X_0}\!\left(\frac{R}{2}e^{-t}\right)
    +2e^{2t}L_{X_0}(0)
    L_{\widehat G_{a/2}}\!\left(\frac{R}{2}e^{-t}\right),
    & R>2e^t.
    \end{cases}
\end{equation}
Combining this definition with \eqref{A.7}, we obtain a nonnegative function $l'$ such that
$l'(R)\to0$ as $R\to\infty$ and
\begin{equation}
    L_Z(R)\leq l'(R).
\end{equation}
Denote 
 \begin{equation}
	p^{*}_{t}f  \triangleq f_{p^{*}_{t}X}                   %
 \end{equation}
where  $ f_{p^{*}_{t}X}    $ is the density function of $p^{*}_{t}X$.
Following \cite{ma2024entropic}, define 
\begin{equation}
	\mathcal{H}_{t,l} = \lbrace  p^{*}_{t}f: \,\mathbb{E}f=0, \, L_{f} \leq l,\, J(f) <\infty,\, f\in \mathcal{T} \rbrace
\end{equation}
where  $\mathcal{T}$ is the set of all absolutely continuous density functions. Therefore, 
\begin{equation}
	\mathcal{I}_{a,l} \subset \mathcal{H}_{t,l^{'}}
\end{equation}
Note that  $\overline{\mathcal{I}_{a,l} } $ is a closed set and 
 $H_{t,l'}$ is compact in $L^1_{1+x^2}$ by \cite{ma2024entropic},
so $\overline{\mathcal{I}_{a,l} } $ is also a compact set in $L^{1}(\mathbb{R},1+x^{2}dx)$. Finally, since 
\begin{equation}\label{5.11}
		\Vert \cdot \Vert_{L^{1}} \leq 	\Vert \cdot \Vert_{L^{1}_{1+x^{2}}}
\end{equation}
where $ L^{1} =L^{1}(\mathbb{R},dx) $ and $L^{1}_{1+x^{2}} = L^{1}(\mathbb{R},1+x^{2}dx)$, then 
$\overline{\mathcal{I}_{a,l} } $ is also a compact set in $L^{1}$.

Next, we prove the $L^1$-continuity of the Fisher information on $\overline{\mathcal{I}_{a,l} } $.
Let $p_n,p\in \overline{\mathcal{I}_{a,l} } $ be density functions such that
\begin{equation}
    \Vert p_{n} - p\Vert_{L^{1}}  \rightarrow 0,\, n\rightarrow \infty
\end{equation}
We shall prove that $J(p_n)\to J(p)$.
For each $n$, let $X_n$ be a random variable with density $p_n$ and without loss of generality, assume that $p_{n} \in \mathcal{I}_{a,l} $. 
 
First we show that there exists a subsequence $\{p_{n_{k}}\}$ such that $p^{'}_{n_{k}} \overset{\text{a.e.}}{\longrightarrow}
p^{'}$ where $p^{'} $ denotes the derivative of the density function $p$. Since $	\mathcal{I}_{a,l} \subset \mathcal{H}_{t,l^{'}}$
and $\mathcal{H}_{t,l^{'}}$ is a compact set, we can find a function $q \in \mathcal{H}_{t,l^{'}}$ and 
a subsequence $p_{n_{k}}$ such that 
\begin{equation}\label{5.12}
	p_{n_{k}} \overset{L^{1}_{1+x^{2}}}{\longrightarrow}  q, \, n_{k}\rightarrow \infty.
\end{equation}
Since $	p_{n} \overset{L^{1}}{\longrightarrow}  p$,  we obtain $p=q $ a.e.. Denote $X_{n} \stackrel{d}{=} 
p^{*}_{t}Z_{n}$. Let $f_n$ denote the density of $e^{-t}Z_n$, and let $\phi_t$ denote the density of $\sqrt{1-e^{-2t}}G$. Then it follows from equation (\ref{5.12}) that 
\begin{equation}\label{5.13}
	f_{n_{k}}\ast \phi_{t} \overset{L^{1}_{1+x^{2}}}{\longrightarrow} p,\, n_{k}\rightarrow \infty,
\end{equation}
where $f\ast g$ denotes the  convolution of f and g. Note $e^{-t}Z_{n_{k}} =X_{0,n_{k}}+\hat{G}_{\frac{a}{2}}$
so we can find $\tilde{t}$ and a function $\tilde{l} $ decreasing to 0 such that 
\begin{equation}
	e^{-t}Z_{n_{k}} \in \mathcal{H}_{\tilde{t},\tilde{l}}.
\end{equation}
Since $\mathcal{H}_{\tilde{t},\tilde{l}}$  is a compact set, we can also extract a 
subsequence such that 
\begin{equation}\label{A.19}
	f_{n_{k_{j}}} \overset{L^{1}_{1+x^{2}}}{\longrightarrow} f,  \, n_{k_{j}}\rightarrow \infty,
\end{equation}
where $f \in \mathcal{H}_{\tilde{t},\tilde{l}}$. For any $y \in \mathbb{R} $,
\begin{equation}\label{5.16}
		\vert(f_{n_{k_{j}}}-f) \ast \phi_{t} (y) \vert = \vert\int_{\mathbb{R}} (f_{n_{k_{j}}}(x)-f(x))  
	\phi_{t}(y-x) dx \vert\leq C_{t} \Vert f_{n_{k_{j}}} -f \Vert_{L^{1}},
\end{equation}
where $C_{t} = \frac{1}{\sqrt{2\pi (1-e^{-2t})}}$.
From equations (\ref{5.11}) and (\ref{A.19}), we have $\Vert f_{n_{k_{j}}} -f \Vert_{L^{1}} \rightarrow 0, n_{k_{j}}\rightarrow \infty$.
So $f_{ n_{k_{j}}}\ast \phi_{t}  \overset{\text{a.e.}}{\longrightarrow} f\ast \phi_{t},n_{k_{j}}\rightarrow \infty$. And by equation (\ref{5.13}),
we have $p = f\ast \phi_{t}$ a.e.. Since 
\begin{equation}
\begin{aligned}
    \left| p'_{n_{k_j}}(y)-p'(y)\right|
    &=
    \left| 
    \int_{\mathbb{R}} 
    \frac{y-x}{1-e^{-2t}}
    \left(f_{n_{k_j}}(x)-f(x)\right)
    \phi_t(y-x)\,dx
    \right| \\
    &\leq
    C'_t |y|
    \int_{\mathbb{R}}
    \left|f_{n_{k_j}}(x)-f(x)\right|\,dx
    +
    C'_t
    \int_{\mathbb{R}}
    |x|
    \left|f_{n_{k_j}}(x)-f(x)\right|\,dx \\
    &\leq
    C'_t |y|
    \left\|f_{n_{k_j}}-f\right\|_{L^1}
    +
    C'_t
    \left\|f_{n_{k_j}}-f\right\|_{L^1_{1+x^2}},
    \qquad \forall y\in\mathbb{R}.
\end{aligned}
\end{equation}
where $C_{t}^{'}=\frac{1}{(2\pi)^{\frac{1}{2}}(1-e^{-2t})^{\frac{3}{2}}}$. Hence,
 by \eqref{A.19}, we obtain 
\begin{equation}\label{A.23}
	p^{'}_{n_{k_{j}}} \overset{\text{a.e.}}{\longrightarrow}
p^{'}, \, n_{k_{j}} \rightarrow \infty.
\end{equation}
Secondly, since $\Vert p_{n} - p\Vert_{L^{1}}  \rightarrow 0, n\rightarrow \infty $, we can find a subsequence
such that $p_{n_{k}} \overset{\text{a.e.}}{\longrightarrow} p$. By \eqref{A.23}, we can 
further find a subsequence by the first step such that 
$p^{'}_{n_{k_{j}}} \overset{\text{a.e.}}{\longrightarrow}
p^{'}, \, n\rightarrow \infty$.  Using $(\sqrt{p})^{'}=\frac{p^{'}}{2\sqrt{p}}$, we obtain
\begin{equation}
	(\sqrt{p_{n_{k_{j}}}})^{'} \overset{\text{a.e.}}{\longrightarrow} (\sqrt{p})^{'}, \, n_{k_{j}}\rightarrow \infty,
\end{equation}
Thirdly,  let $p_{n_{k}}$ be a subsequence such that $(\sqrt{p_{n_{k}}})^{'} \overset{\text{a.e.}}{\longrightarrow} (\sqrt{p})^{'}$.
Since $p_{n_{k}} \in \mathcal{H}_{t,l^{'}}$ and by Proposition 5.8 of \cite{ma2024entropic}, we have 
\begin{equation}
	((\sqrt{p_{n_{k}}})^{'})^{2} = \frac{(p^{'}_{n_{k}})^{2}}{4p_{n_{k}}} \leq B_{t}p^{*}_{2t}(e^{t}Z_{n_{k}}),
\end{equation}
where $B_{t}> 0 $ is a constant. Note that $p^{*}_{2t}(e^{t}Z_{n_{k}})=f_{n_{k}}\ast \phi_{2t}$ and  
\begin{align}
	\int_{\mathbb{R}}\vert(f_{n_{k}}-f) \ast \phi_{2t}\vert (y) dy &= \int_{\mathbb{R}}
    \vert\int_{\mathbb{R}}
	(f_{n_{k}}(x)-f(x))  \phi_{2t}(y-x) dx \vert dy \\
	&\leq \Vert f_{n_{k}}-f\Vert_{L^{1}}, \\
\vert(f_{n_{k}}-f) \ast \phi_{2t}\vert (y) &\leq  \vert\int_{\mathbb{R}} (f_{n_{k}}(x)-f(x))  
	\phi_{2t}(y-x) dx \vert\leq C_{t}^{''} \Vert f_{n_{k}} -f \Vert_{L^{1}},
\end{align}
where $C^{''}_{t}=\frac{1}{\sqrt{2\pi (1-e^{-4t})}}$. Therefore by 
 \eqref{A.19} we  can extract a further subsequence such that
\begin{equation}
	f_{n_{k_j}}*\phi_{2t}\to f*\phi_{2t}
\quad\text{both a.e. and in }L^1.
\end{equation}
Then by the dominated convergence theorem, we obtain
\begin{equation}
	(\sqrt{p_{n_{k_{j}}}})^{'} \overset{L^{2}}{\longrightarrow}(\sqrt{p})^{'}, \,  n_{k_{j}}\rightarrow \infty.
\end{equation}
Finally, we have 
\begin{equation}
	J(p_{n_{k_{j}}}) \rightarrow J(p), \,  n_{k_{j}}\rightarrow \infty,
\end{equation}
due to $J(p_{n_{k_{j}}}) = 4 \Vert (\sqrt{p_{n_{k_{j}}}})^{'}   \Vert^{2}_{L^{2}} $. Therefore, for any 
subsequence $\{ J(p_{n_{k}})\}_{k \in \mathbf{N}}$, we can extract a further subsequence  such that 
$	J(p_{n_{k_{j}}}) \rightarrow J(p), \,  n_{k_{j}}\rightarrow \infty$. This proves that 
\begin{equation}
	J(p_{n}) \rightarrow J(p), \,  n\rightarrow \infty.
\end{equation}

\end{proof}

\begin{acks}
We thank Huazi Zhang for providing this research direction. We thank Zhi-Ming Ma
for his helpful comments.
\end{acks}

\end{document}